\documentclass[11pt]{articlefederico}
\usepackage{graphicx}
\usepackage{amsfonts}
\usepackage{amsmath}
\usepackage{fullpage}
\usepackage{amssymb}
\usepackage{amsthm}
\usepackage{mathtools}
\usepackage{tikz}
\usepackage{lipsum}
\usepackage{mathrsfs}
\usepackage{float}
\usepackage{multicol}
\usepackage{bbm}

 \usepackage[hidelinks,backref=page]{hyperref} 
 \hypersetup{
    colorlinks,
    citecolor=magenta,
    filecolor=magenta,
    linkcolor=blue,
    urlcolor=black
}

\usepackage{cleveref}

\newtheorem{thm}{Theorem}[section]

\newtheorem{conj}[thm]{Conjecture}
\newtheorem{defin}[thm]{Definition}
\newtheorem{lm}[thm]{Lemma}
\newtheorem{smpl}[thm]{Example}
\newtheorem{rem}[thm]{Remark}

\newcommand{\w}{\mathsf{w}}
\newcommand{\x}{\mathsf{x}}
\newcommand{\y}{\mathsf{y}}
\newcommand{\e}{\mathsf{e}}

\newcommand{\R}{\mathbb{R}}
\newcommand{\Z}{\mathbb{Z}}
\renewcommand{\L}{\mathsf{L}}

\newcommand{\FF}{\mathcal{F}}
\newcommand{\GG}{\mathcal{G}} 

\newcommand{\1}{\mathbbm{1}}

\newcommand{\splamb}{\boldsymbol{\lambda}}
\newcommand{\spmu}{\boldsymbol{\mu}}
\newcommand{\sppi}{\boldsymbol{\pi}}

\newcommand{\cl}{\mathrm{cl}}
\newcommand{\opi}{{\boldsymbol{\pi}}}
\newcommand{\otau}{{\boldsymbol{\tau}}}

\begin{document}

\title{\textsf{The tropical critical points of an affine matroid}}

\author{
\textsf{Federico Ardila-Mantilla\footnote{\noindent \textsf{San Francisco State University, Universidad de Los Andes. federico@sfsu.edu. Supported by NSF Grant DMS-2154279.}}} 
\\ \and 
\textsf{Christopher Eur\footnote{\noindent \textsf{Harvard University; ceur@math.harvard.edu. Supported by  NSF Grant DMS-2001854.}}}  
\\ \and 
\textsf{Raul Penaguiao\footnote{\noindent \textsf{Max Planck Institute, Leipzig; 
raul.penaguiao@mis.mpg.de. Supported by  SNF Grant P2ZHP2 191301.}}}
}
%

\date{}

\maketitle

\begin{abstract}
We prove that the number of tropical critical points of an affine matroid $(M,e)$ is equal to the beta invariant of $M$. Motivated by the computation of maximum likelihood degrees, this number is defined to be the degree of the intersection of the Bergman fan of $(M,e)$ and the inverted Bergman fan of $N=(M/e)^\perp$, where $e$ is an element of $M$ that is neither a loop nor a coloop. Equivalently, for a generic weight vector $w$ on $E-e$, this is the number of ways to find weights $(0,x)$ on $M$ and $y$ on $N$ with $x+y=w$ such that on each circuit of $M$ (resp. $N$), the minimum $x$-weight (resp. $y$-weight) occurs at least twice. This answers a question of Sturmfels. 
\end{abstract}

\section{\textsf{Introduction}}

During the Workshop on Nonlinear Algebra and Combinatorics from Physics at the Center for the Mathematical Sciences and Applications at Harvard University in April 2022, Bernd Sturmfels \cite{Sturmfels2022} posed one of those combinatorial problems that is deceivingly simple to state, but whose answer requires a deeper understanding of the objects at hand. 


\begin{conj}\label{problem} \cite{Sturmfels2022} 
Let $M$ be a matroid on $E$, and let $e\in E$ be an element that is neither a loop nor a coloop. Let $M/e$ be the contraction of $M$ by $e$ and let  $N=(M/e)^\perp$ be its dual matroid.
\begin{enumerate}
\item
\emph{(Combinatorial version)}
Given a vector $\w \in \R^{E-e}$, we wish to find weight vectors
 $(0,\x) \in \R^E$ on $M$ (where $e$ has weight $0$) and $\y \in \R^{E-e}$ on $N$ such that

$\bullet$ on each circuit of $M$, the minimum $\x$-weight occurs at least twice,

$\bullet$ on each circuit of $N = (M/e)^\perp$, the minimum $\y$-weight occurs at least twice, and

$\bullet$ $\w=\x+\y$.

\noindent For generic $\w$, the number of solutions is the beta invariant $\beta(M)$. 
\item
\emph{(Geometric version)}
The degree of the stable intersection of the Bergman fan $\Sigma_{(M,e)}$ and the inverted Bergman fan $-\Sigma_N=-\Sigma_{(M/e)^\perp}$ is
\[
\deg\left( \Sigma_{(M,e)} \cdot -\Sigma_{(M/e)^\perp} \right)= \beta(M)\, .
\]
\end{enumerate}
\end{conj}

The goal of this paper is to prove this conjecture.

\begin{thm}  \label{thm:main} 
Versions 1. and 2. of \Cref{problem} are true.
\end{thm}

The affine matroid $(M,e)$ is the matroid $M$ with a special chosen element $e$. 
The Bergman fan of $(M,e)$ is the Bergman fan of $M$ intersected with the hyperplane $x_e=0$. The other relevant definitions are given in Section \ref{sec:tropgeom}.
The combinatorial and geometric formulations of \Cref{problem} are equivalent because in the stable intersection above, all intersection points have multiplicity 1 \cite[Lemma 7.4]{agostini2021likelihood}.

\begin{figure}[h]
\begin{center}
\includegraphics[scale=.9]{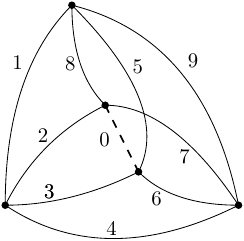}
\qquad
\includegraphics[scale=.9]{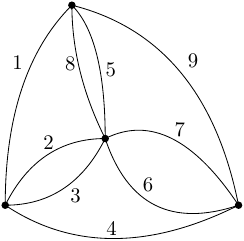}
\qquad
\includegraphics[scale=.8]{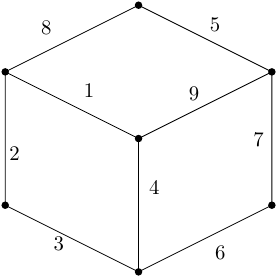}
\end{center}
\caption{A graph $G$, its contraction $G/0$, and its dual $H=(G/0)^\perp$.\label{./img:dualmatroid}}
\end{figure}

The results of this paper were motivated by the problem of computing maximum likelihood degrees in algebraic statistics, pioneered by Catanese, Khetan, Ho\c{s}ten and Sturmfels \cite{Catanese}. For linear models, Varchenko showed that the maximum likelihood degree equals the beta invariant of the corresponding matroid; see \cite{Varchenko, Zaslavsky} and  \cite[Theorem 13]{Catanese}.

%
%
%
%
%
%

Agostini, Brysiewicz, Fevola, K\"uhne, Sturmfels, and Telen  first encountered a special case of \Cref{problem} in \cite{agostini2021likelihood}.
Using algebro-geometric results of Huh and Sturmfels \cite{HuhSturmfels}, which built on earlier work of Varchenko \cite{Varchenko}, they proved \Cref{thm:main}  for matroids realizable over the real numbers \cite[Theorem 7.1]{agostini2021likelihood}. 
In a related setting of linear Gaussian models, the maximum likelihood degrees were shown to be matroid invariants of the linear subspace \cite{SU10, EFSS21}.

We prove \Cref{thm:main} for all matroids.
Following the original motivation, we call the solutions to \Cref{problem}.1 the \emph{tropical critical points} of the affine matroid; our main result is that they are counted by Crapo's beta invariant $\beta(M)$. We do something stronger. Agostini et.\ al.\ write 
\begin{quote}
``we would like to describe the multivalued map that takes any tropical data vector $\w$ to the set of its critical points". \cite[Section 7]{agostini2021likelihood}
\end{quote}
We give an explicit formula for this map for all $\w$ that are rapidly increasing, under any order $<$ on the ground set $E$.

In Section \ref{sec:3} we prove  \Cref{thm:main}.1 combinatorially, relying on the tropical geometric fact that the number of solutions is the same for all generic\footnote{We will say that a property holds for generic $\w \in \R^n$ if it holds for all $\w$ outside of a polyhedral complex of dimension smaller than $n$.}
$\w$. We show that when the entries of $\w$ are rapidly increasing with respect to some order $<$ on $E$, the solutions to \Cref{problem}.1 are naturally in bijection with the $\beta$-nbc-bases of the matroid with respect to $<$. It is known that the number of such bases is the beta invariant of the matroid, regardless of the order $<$.

In Section \ref{sec:4} we sketch a proof of \Cref{thm:main}.2 that relies the theory of \emph{tautological classes of matroids} of Berget, Eur, Spink, and Tseng \cite{BEST}.
This proof is not combinatorial; it relies on computations in the equivariant Chow ring of the permutahedral variety initiated in \cite{BEST} and extended here.
We do not know of a direct relationship between our two proofs.  For a survey of the relationships and differences between these proof techniques in a similar setting, see \cite{Ard24}.

\section{\textsf{Notation and preliminaries\label{sec:2}}}

\subsection{\textsf{The lattice of set partitions}}

A set partition $\splamb$ of a set $E$ is a collection of subsets, called blocks,  of $E$, say $\splamb = \{\lambda_1 , \dots , \lambda_{\ell}\}$, whose union is $E$ and whose pairwise intersections are empty.
We write $\splamb \models E$. We let $|\splamb| = \ell$ be the number of blocks of $\splamb$.
If $e\in E$ and $\splamb \models E$, we write $\splamb(e)$ for the block of $\splamb $ that contains $e$.

We define the \emph{linear space of a set partition $\splamb = \{\lambda_1, \ldots, \lambda_\ell\} \models E$} to be
\begin{eqnarray*}
 \L (\splamb ) &\coloneqq&  \textrm{span}\{{\e}_{\lambda_1}, \ldots, {\e}_{\lambda_\ell}\} \, \subseteq \R^E  \\ 
 &=& 
\{ \x \in \R^E  \, \,  | \,  x_i = x_j \text{ whenever } i, j \text{ are in the same block of } \splamb\},
  \end{eqnarray*}
where $\{\e_i \, : \, i \in E\}$ is the standard basis of $\R^E$ and $\e_S = \sum_{s \in S} e_s$ for $S \subseteq E$.
Notice that $\dim \L (\splamb) = |\splamb|$.
The map $\splamb \mapsto \L(\splamb)$ is a bijection between the set partitions of $E$ and the flats of the 
\emph{braid arrangement}, which is the hyperplane arrangement in $\R^E$
given by the hyperplanes $x_i=x_j$ for $i \neq j$ in $E$. 
%

If $e \in E$ then we write 
$\L(\splamb)|_{x_e=0} = \{\x \in \R^{E-e} \, : \, (0, \x) \in \L(\splamb) \subseteq \R^E\}$.

\subsection{\textsf{The intersection graph of two set partitions}}

We denote $[a, b] \coloneqq \{a, a+1, \ldots, b-1, b\}$ and $[n] \coloneqq [1, n]$.
The following construction from \cite{ArdilaEscobar} will play an important role.

\begin{defin}
Let $\splamb \models [0,n]$ and $\spmu \models [n]$ be set partitions. 
The \emph{intersection graph} $\Gamma = \Gamma_{\splamb,\spmu}$
is the bipartite graph
with vertex set $\splamb \sqcup \spmu$ and edge set $[n]$, where the edge labelled $e$ connects the parts $\lambda(e)$ of $\splamb$ and $\mu(e)$ of $\spmu$ containing $e$. 
The vertex  $\lambda(0)$ is marked with a hollow point. 
\end{defin}

 \begin{figure}[h]
 \begin{center}
 \includegraphics[scale=1.2]{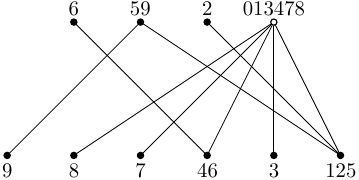} \qquad 
 \includegraphics[scale=1.3]{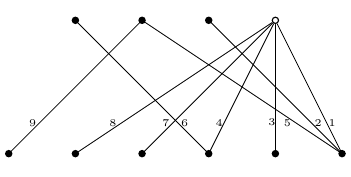} 
 \end{center}
 \caption{The intersection graph of $\splamb = \{6,59,2,013478\}\models [0,9]$ and $\spmu = \{9,8,7,46,3,125\}\models [9]$. 
 We omit brackets for legibility.
 \textbf{Left:} The vertices are labelled by the blocks of the set partitions.  \textbf{Right:} The edges are labelled by the elements of $[9]$. 
\label{fig:intgraph}}	
\end{figure}

The intersection graph may have several parallel edges connecting the same pair of vertices. Notice that the label of a vertex in $\Gamma$ is just the set of labels of the edges incident to it. Therefore we can remove the vertex labels, and simply think of $\Gamma$ as a bipartite multigraph on edge set $[n]$. This is illustrated in Figure \ref{fig:intgraph}.

\begin{lm}\label{lm:arboreal}
Let $\splamb \models [0,n]$ and $\spmu \models [n]$ be set partitions and $\Gamma_{\splamb, \spmu}$ be their intersection graph.
\begin{enumerate}
\item
If $\Gamma_{\splamb, \spmu}$ has a cycle, then $\L(\splamb)|_{x_0=0} \cap (\w - \L(\spmu)) = \emptyset$ for generic
$\w \in \R^n$.
\item
If $\Gamma_{\splamb, \spmu}$ is disconnected, then $\L(\splamb)|_{x_0=0} \cap (\w - \L(\spmu))$ is not a point for any $\w \in \R^n$.
\item
If $\Gamma_{\splamb, \spmu}$ is a tree, then $\L(\splamb)|_{x_0=0} \cap (\w - \L(\spmu))$ is a point for any $\w \in \R^n$.
\end{enumerate}
\end{lm}

\begin{proof}
Let $\x \in \L(\splamb) $ and $\y \in \L(\spmu)$ such that $\x + \y = \w$.
Write $x_{\lambda(i)} \coloneqq x_i$ and $y_{\mu(i)} \coloneqq y_i$ for $i \in [n]$.
The subspace 
$\L(\splamb)|_{x_0=0} \cap (\w - \L(\spmu))$, 
can be naturally regarded as living in $\R^{\splamb \sqcup \spmu}$, where it 
is cut out by the equalities 
\begin{eqnarray*}
x_{\lambda(i)} + y_{\mu(i)} &=& w_i \qquad \textrm{ for } i \in [n], \\
x_{\lambda(0)} &=& 0.
\end{eqnarray*}

This system has $n+1$ equations and $|\splamb| + |\spmu|$ independent unknowns.
The linear dependences among these equations are controlled by the cycles of the graph $\Gamma_{\splamb, \spmu}$. More precisely, the first $n$ linear functionals $\{x_{\lambda(i)} + y_{\mu(i)} \, : \, i \in [n] \}$  on $\R^{\splamb \sqcup \spmu}$  give a realization of the graphical matroid of $\Gamma_{\splamb, \spmu}$. The last equation is clearly linearly independent from the others.

If $\Gamma_{\splamb, \spmu}$ has a cycle with edges $i_1, i_2, \ldots, i_{2k}$ in that order, then the above equalities imply that $w_{i_1} - w_{i_2} + w_{i_3} - \cdots - w_{i_{2k}} = 0$. For generic $\w$ this equation does not hold, so we have $\L(\splamb)|_{x_0=0} \cap (\w - \L(\spmu))= \emptyset$.

If $\Gamma_{\splamb, \spmu}$ is disconnected, let $A$ be the set of edges in a connected component not containing the vertex $\lambda(0)$. If $\x \in \L(\splamb)$ and $\y \in \L(\spmu)$ satisfy $\x+\y= \w$ and $x_0=0$, then $\x+r\e_A \in \L(\splamb)$ and $\y-r\e_A \in \L(\spmu)$ also satisfy those equations for any real number $r$. Therefore $\L(\splamb)|_{x_0=0} \cap (\w - \L(\spmu))$ is not a point.

Finally, if $\Gamma_{\splamb, \spmu}$ is a tree, then its number of vertices is one more than the number of edges, that is, $n + 1 = |\splamb| + |\spmu|$, so the system of equations has equally many equations and unknowns. Also, these equations are linearly independent since $\Gamma_{\splamb, \spmu}$ is a tree. It follows that the system has a unique solution.
\end{proof}

When $\Gamma_{\splamb, \spmu}$ is a tree, we call $\splamb$ and $\spmu$ an \emph{arboreal pair}. 

\begin{lm}\label{lm:x,y}
Let $\splamb \models [0,n]$ and $\spmu \models [n]$ be an arboreal pair of set partitions and let $\Gamma_{\splamb, \spmu}$ be their intersection tree. Let $\w \in \R^n$. The unique vectors $\x \in \L(\splamb)$ and $\y \in \L(\spmu)$ such that $\x + \y = \w$ and $x_{0}=0$ are given by
\begin{eqnarray*}
x_{\lambda_i} &=& w_{e_1}-w_{e_2} + \cdots \pm w_{e_k} \qquad \textrm{ where $e_1e_2\ldots e_k$ is the unique path from $\lambda_i$ to $\lambda(0)$} \\
y_{\mu_j} &=& w_{f_1}-w_{f_2} + \cdots \pm w_{f_l} \qquad \textrm{ where $f_1f_2\ldots f_l$ is the unique path from $\mu_j$ to $\lambda(0)$}
\end{eqnarray*}
for any $i$ and $j$.
\end{lm}

\begin{proof}
This follows readily from the fact that, for each $1 \leq i \leq k$, the values of $x_{\lambda(e_i)}$ and $y_{\mu(e_i)}$ on the vertices incident to edge $i$ have to add up to $w_{e_i}$.
\end{proof}

\begin{smpl}
Let $\splamb = \{6,59,2,013478\}\models [0,9]$ and $\spmu = \{9,8,7,46,3,125\}\models~[9]$. These set partitions form an arboreal pair, as evidenced by their intersection tree, shown in Figure \ref{fig:intgraph}. We have, for example, $y_9 = w_9-w_5+w_1$ because the path from $\mu(9)=\{9\}$ to $\lambda(0)=\{013478\}$ uses edges $9,5,1$ in that order. The remaining values are:
\begin{eqnarray*}
&&x_6 = w_6-w_4, \quad x_{59} = w_5-w_1, \quad x_2 = w_2-w_1, \quad x_{013478} = 0,\\
&&y_9 = w_9-w_5+w_1, \quad y_8=w_8, \quad y_7 = w_7, \quad y_{46}=w_4, \quad y_3=w_3, \quad y_{125} = w_1.
\end{eqnarray*}
\end{smpl}

The tropical critical points of a matroid are better behaved for the following family of vectors.

\begin{defin}\label{def:super}
A vector $\w \in \R^{n}$ is \emph{rapidly increasing} if $w_{i+1} > 3 w_i > 0$ for $1 \leq i \leq n-1$.
\end{defin}

The next lemma is readily verified.

\begin{lm}\label{lm:omega}
Let $\w$ be rapidly increasing. For any 
 $1 \leq a < b \leq n$ and any choice of $\epsilon_i$s and $\delta_i$s in $\{-1,0,1\}$, we have
 $
w_a + \sum_{i=1}^{a-1}  \epsilon_i w_i < 
w_b + \sum_{j=1}^{b-1}  \delta_j w_j.
$
\end{lm}
%

\begin{defin}
Given a rapidly increasing vector $\w \in \R^{n}$ and a real number $x$, we will say $x$ \emph{is near} $w_i$ and write $x \approx w_i$ if  $w_i - (w_1+\cdots + w_{i-1}) \leq  x \leq w_i+(w_1+ \cdots + w_{i-1})$ for $i = 1, \ldots, n$. 
By Lemma \ref{lm:omega}, if $x \approx w_i$ and $y \approx w_j$ for $i<j$ then $x<y$.
\end{defin}

\subsection{\textsf{Matroids, Bergman fans, and tropical geometry}\label{sec:tropgeom}}

In what follows we will assume familiarity with basic notions in matroid theory; for definitions and proofs, see \cite{Oxley, Welsh}. 
We also state here some facts from tropical geometry that we will need; see \cite{maclagan2015introduction, mikhalkin2010tropical} for a thorough introduction.

Let $M$ be a matroid on $E$ of rank $r+1$. The \emph{dual matroid} $M^{\perp}$ is the matroid on $E$ whose set of bases is $\{B^\perp \, | \, B \text{ is a basis of $M$}\}$, where $B^\perp \coloneqq E-B$.
The following lemma is useful to see how $M$ and $M^\perp$ interact; see \cite[Lemma 3.14]{ADH1} and  \cite[Proposition 2.1.11]{Oxley} for proofs.

\begin{lm}
\label{lm:unionflat_coflat}
If $F$ is a flat of $M$ and $G$ is a flat of $M^\perp$, then $|F \cup G| \neq |E| - 1$.
\end{lm}

\begin{defin} \cite{Crapo}
The \emph{beta invariant} of $M$ is defined to be $\beta(M) \coloneqq (-1)^r \frac{d\chi_M (t)}{d t} \Big|_{t=1}$\, , 
where $\chi_M(t)$ is the \emph{characteristic polynomial} of $M$:
\[
\chi_M(t) \coloneqq \sum_{X \subseteq E} (-1)^{|X|}t^{r(M) - r(X)}.
\]
\end{defin}

\begin{defin}
Fix a linear order $<$ on $M$. 
A \emph{broken circuit} is a set of the form $C - \text{min}_< C$ where $C$ is a circuit of $M$. An \emph{nbc-basis} of $M$ is a basis of $M$ that contains no broken circuits. A \emph{$\beta$-nbc-basis} of $M$ is an nbc-basis $B$ such that $B^\perp \cup 0 \setminus 1$ is an nbc-basis of $M^\perp$.
\end{defin}

\begin{thm} \cite{Ziegler} For any linear order $<$ on $E$, the number of $\beta$-nbc-bases of $M$ is equal to the beta invariant $\beta(M)$.
 \end{thm}

The \emph{closure} of a set $A \subseteq E$, denoted by $\cl_M(A)$, is the smallest flat $F$ containing $A$.
For each basis $B = \{b_1 > \dots > b_r>b_{r+1}\}$ of the matroid $M$, we define the complete flag of flats
\[
\FF_M(B) \coloneqq \{\emptyset  \subsetneq \cl_M \{b_1\} \subsetneq \cl_M \{b_1, b_2\} \subsetneq \cdots \subsetneq \cl_M \{b_1, \dots, b_r\}  \subsetneq E\}.
\]

The following characterization of nbc-bases will be useful.

\begin{lm}\cite[(7.30), (7.31)]{Bjorner} \label{lm:nbc_flats}
Let $M$ be a matroid of size $n+1$ and rank $r+1$, and $B$ a basis of $M$.
Then $B$ is an nbc-basis of $M$ if and only if $b_i = \min F_i$ for $i = 1, \dots, r+1$.
\end{lm}

An \emph{affine matroid} $(M,e)$ on $E$ is a matroid $M$ on $E$ with a chosen element $e \in E$ \cite{Ziegler}.
The set $E$ is also called the ground set of $(M, e)$.

\begin{defin}\cite{Sturmfels} \label{defin:bergfan}
The \emph{Bergman fan} of a matroid $M$ on $E$ is
\[
\Sigma_M = \{\x \in \R^E \, | \, \min_{c\in C} x_c \text{ is attained at least twice for any circuit $C$ of $M$} \}\, .
\]
The \emph{Bergman fan of an affine matroid} $(M,e)$ on $E$ is
\[
\Sigma_{(M,e)} = \{\x \in \R^{E-e} \, | \, (0, \x) \in \Sigma_M\} = \Sigma_M|_{x_e=0}.
\]
\end{defin}

\begin{rem}\label{rem:proj}
The Bergman fan contains the lineality space $\1\R$.
Taking the quotient by this space, or intersecting with a coordinate linear hyperplane will give the same result, and typically the (projective) Bergman fan is defined in the quotient vector space $\R^E/{\1\R}$ in the literature.
\end{rem}


The motivation for this definition comes from tropical geometry.
A subspace $V \subset \R^E$ determines a matroid $M_V$ on $E$, and the tropicalization of $V$ is precisely the Bergman fan of $M_V$. Similarly, an affine subspace $W \subset \R^{E-e}$ determines an affine matroid $(M_W,e)$ on $E$, where $e$  represents the hyperplane at infinity. The tropicalization of $W$ is the Bergman fan $\Sigma_{(M_W,e)}$.

\begin{thm}\label{thm:bergman}\cite{ArdilaKlivans}
The Bergman fan of a matroid $M$ is equal to the union of the cones 
\begin{eqnarray*}
\sigma_\FF &=&  \textrm{cone}(\e_{F_1}, \ldots, \e_{F_{r+1}}) + \R\1\\ 
&=& \{ \x \in \R^E \, | \, x_a \geq x_b \text{ whenever $a\in F_i$ and $b\in F_j$ for some $1 \leq i \leq j \leq r+1$}\} \end{eqnarray*}
for the complete flags $\FF = \{\emptyset = F_0 \subsetneq F_1 \subsetneq \dots \subsetneq F_r \subsetneq F_{r+1} = E\}$ of flats of $M$.
It is a tropical fan with weights $w(\FF) = 1$ for all $\FF$.
\end{thm}

If $\Sigma_1$ and $\Sigma_2$ are tropical fans of complementary dimensions, then $\Sigma_1$ and $v+\Sigma_2$ intersect transversally at a finite set of points for any sufficiently generic vector $v \in \R^n$. Furthermore, each intersection point $p$ is equipped with a multiplicity $w(p)$ that depends on the respective intersecting cones, in such a way that the quantity
\[
\deg (\Sigma_1 \cdot \Sigma_2) := \sum_{p \in \Sigma_1 \cap (v+\Sigma_2)} w(p)
\]
is constant for generic $v$ \cite[Proposition 4.3.3, 4.3.6]{mikhalkin2010tropical};
this is called the \emph{degree} of the intersection.

In all the tropical intersections that arise in this paper, it was verified in \cite[Lemma 7.4]{agostini2021likelihood} that the multiplicity index $w(p)$ is 1. This also follows readily from the fact that every such intersection comes from an arboreal pair $\splamb$, $\spmu$ by \Cref{lm:arboreal}, as explained in the next section. 
Therefore the degree of the intersection will be simply the number of intersection points:
\[
\deg (\Sigma_{(M,e)} \cdot - \Sigma_{(M/e)^\perp}) = | \Sigma_{(M,e)} \cap  (v - \Sigma_{(M/e)^\perp})|
\]
for generic $v \in \R^{E-e}$. This explains the equivalence of the two versions of \Cref{problem} and \Cref{thm:main}.

%

\section{\textsf{Proof of the main theorem via basis activities \label{sec:3}}}

Let $M$ be a matroid on $[0,n]$ of rank $r+1$ such that $0$ is not a loop nor a coloop. Then $M/0$ has rank $r$, and $N=(M/0)^\perp$ has rank $n-r$.
For any basis $B$ of $M$ containing $0$, $B^\perp = [0,n]-B $ is a basis of $N  = (M/0)^\perp$. Conversely, every basis of $N$ equals $B^\perp$ for a basis $B$ of $M$ containing $0$.

Let us construct an intersection point in  $\Sigma_{(M,0)} \cap (\w - \Sigma_N)$ for each $\beta$-nbc-basis  of $M$.

\begin{lm}\label{lm:bnbc_intersect}
Let $M$ be a matroid on $E=[0, n]$ of rank $r+1$ such that $0$ is not a coloop, and let $N = (M/0)^{\perp}$. Let $\w \in \R^{n}$ be rapidly increasing. 
For any $\beta$-nbc-basis $B$ of $M$, 
there exist 
unique vectors $(0,\x) \in \sigma_{\FF_M(B)}$ and $\y \in \sigma_{\FF_N(B^\perp)}$ such that $\x + \y = \w$.
\end{lm}

\begin{proof}
A flag $\{\emptyset \subsetneq F_1 \subsetneq \cdots \subsetneq F_k \subsetneq E\}$ of subsets of $E$ gives rise to a set partition $\{F_1, F_2-F_1, \ldots, E-F_k\}$ of $E$. First we show that the set partitions $\opi$ and $\opi^\perp$ corresponding to the flags $\FF=\FF_M(B)$ and $\FF^\perp = \FF_N(B^\perp)$ 
form an arboreal pair. Since they have sizes $|B| = r+1$ and $|B^\perp| = n-r$, respectively, their intersection graph has $n+1$ vertices and $n$ edges. Therefore it is sufficient to prove that the intersection graph $\Gamma_{\opi, \opi^\perp}$ is connected; this implies that it is a tree.

Assume contrariwise, and let $A$ be a connected component not containing the edge $1$. Let $a>1$ be the smallest edge in $A$. Then $a$ is the smallest element of its part $\opi(a)$ in $\opi$, and since $B$ is nbc-basis in $M$, this implies $a \in B$. Similarly, 
since $B^\perp$ is nbc-basis in $N$, this also implies $a \in B^\perp$. This  contradicts Lemma \ref{lm:nbc_flats}.

It follows from Lemma \ref{lm:arboreal} that there exist unique $(0,\x) \in \L(\sppi)$ and $\y \in \L(\sppi^\perp)$ such that $\x + \y = \w$. It remains to show that $(0,\x) \in \sigma_\FF$ and $\y \in \sigma_{\FF^\perp}$.

Lemma \ref{lm:x,y} provides formulas for $\x$ and $\y$ in terms of the paths from the various vertices of the tree of $\Gamma_{\sppi,\sppi^\perp}$ to $\pi(0)$. To understand those paths, let us give each edge $e$ an orientation as follows:
\begin{eqnarray*}
\pi(e) \longrightarrow \pi^\perp(e) & \text{ if } &  \min \pi(e) > \min \pi^\perp(e), \\
\pi(e) \longleftarrow \pi^\perp(e) & \text{ if } & \min \pi(e) < \min \pi^\perp(e).
\end{eqnarray*}
We never have $\min \pi(e) = \min \pi^\perp(e)$, because as above, that would imply $e \in B \cap B^\perp$. 

We claim that every vertex other than $\pi(0)$ has an outgoing edge under this orientation. 
Consider a part $\pi_i \neq \pi(0)$ of $\opi$; let $ \min \pi_i = b$. Edge $b$ connects $\pi_i = \pi(b)$ to $\pi^{\perp}(b) \ni b$, and we cannot have $\min \pi^{\perp}(b) > b = \min \pi(b)$, so we must have $\pi_i \rightarrow \pi^{\perp}(b)$. The same argument works for any part $\pi^{\perp}_j$ of $\opi^{\perp}$.

Now, since $B$ is an nbc-basis of $M$, every element $b \in B$ is minimum in $\pi(b)$, so there is a directed path that starts at $\pi(b)$ and can only end at $\pi(0)$, and its first edge is $b$. Furthermore, by the definition of the orientation, the labels of the edges decrease along this path. 
Thus in the alternating sum $x_b = w_b \pm \cdots$ given by Lemma \ref{lm:x,y}, the first term dominates, and $x_b \approx w_b$  for $b \in B \setminus 0$, whereas $x_0=0$. Similarly, since $B^\perp$ is an nbc-basis of $N$, $y_c \approx w_c$ for all $c \in B^\perp$. 

Therefore, if we write $B = \{b_1 > \dots > b_r > b_{r+1}=0\}$, since $w$ is rapidly increasing, it follows that $x_{b_1} > x_{b_2} > \cdots > x_{b_r} > x_{b_{r+1}} = 0$, so from \Cref{thm:bergman} we have $(0,\x) \in \sigma_\FF$.
Similarly, if we write $B^\perp = E-B = \{c_1 > \dots > c_{n-r}=1\}$, then $y_{c_1} > y_{c_2} > \cdots > y_{c_{n-r}}$, so $\y \in \sigma_{\FF^\perp}$. The desired result follows.
\end{proof}

\begin{smpl}\label{smpl:arboreal}
The graphical matroid $M$ of the graph $G$ in \Cref{./img:dualmatroid} has six $\beta$-nbc-bases: $0256$, $0257$, $0259$, $0368$, $0378$, $0379$. Let us compute the intersection point in $\Sigma_{(M,0)} \cap (\w - \Sigma_N)$ associated to $0257$ for the rapidly increasing vector $\w=(10^0, 10^1, \ldots, 10^8) \in \R^9$. 

For $B = 0257$, we have $B^\perp = 134689$. The flags they generate in $M$ and $N$ are
\begin{eqnarray*}
\FF_M(B) &=& \{\emptyset \subsetneq 7  \subsetneq 57  \subsetneq 2457  \subsetneq 0123456789 \} \\
\FF_{N}(B^\perp) &=& \{\emptyset \subsetneq 9\subsetneq 89 \subsetneq 689 \subsetneq 46789 \subsetneq 346789 \subsetneq 123456789 \},
\end{eqnarray*}
which give rise to the corresponding set compositions 
\[
\opi = 7|5|24|013689, \qquad \opi^\perp = 9|8|6|47|3|125.
\]
This is indeed an arboreal pair, as evidenced by their intersection graph in \Cref{./img:arboreal1}.

 \begin{figure}[h]
 \begin{center}
 \includegraphics[scale=1.3]{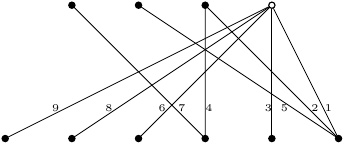}
 \end{center}
 \caption{The intersection graph of $\opi = 7|5|24|13689$ and $\opi^\perp = 9|8|6|47|3|125$.\label{./img:arboreal1}}
 \end{figure}
 
\Cref{lm:bnbc_intersect} gives us the unique points $(0,\x) \in \FF_{\opi}$ and $\y \in \FF_{\otau}$ such that $\x + \y = \w$; they are given by the paths to the special vertex $\pi(0)$  in the intersection tree $\Gamma_{\opi, \opi^\perp}$. For example $x_7 = 10^6-10^3+10^1-10^0 = 999009$ and $y_4 = 10^3-10^1+10^0 = 991$
are given by the paths 7421 and 421 from $\pi(7) = \pi_1$ and $\pi^\perp(7) = \pi^\perp_4$ to $\pi(0)$, respectively. In this way we obtain:
\begin{equation*}
\begin{array}{rrrrrrrrrr}
\x =& 0& 9 & 0 & 9 & 9999 & 0 & 999009 & 0 & 0 \\
\y =& 1& 1& 100& 991& 1& 100000& 991& 10000000& 100000000 \\
\w =& 1& 10& 100& 1000& 10000& 100000& 1000000& 10000000& 100000000
\end{array}
\end{equation*}
and $\x \in \Sigma_{(M,0)} \cap (\w - \Sigma_N)$. We invite the reader to record the weights $(0,\x)$ and $\y$ in the graphs $G$ and $H$ of \Cref{./img:dualmatroid}, and verify that in each cycle the minimum weight appears at least twice.
\end{smpl}

Conversely, the following lemma shows that any intersection point between $\Sigma_{(E,e)}$ and $v - \Sigma_N$ is of the form constructed in \Cref{lm:bnbc_intersect}; that is, it comes from a $\beta$-nbc-basis.

\begin{lm}\label{lm:intersections_are_bnbc}
Let $M$ be a matroid on $E=[0, n]$ of rank $r+1$, such that $0$ is not a loop nor a coloop, and $N = (M/0)^{\perp}$. Let $\w \in \R^n$ be generic and rapidly increasing. Let
\begin{eqnarray*}
\FF &=& \{\emptyset = F_0 \subsetneq F_1 \subsetneq \dots \subsetneq F_r \subsetneq F_{r+1} = E \} \\
\GG &=& \{\emptyset = G_0 \subsetneq G_1 \subsetneq \dots \subsetneq G_{n-r-1} \subsetneq G_{n -r} = E - 0 \}
\end{eqnarray*} 
be complete flags of the matroids $M$ and $N$, respectively, such that 
$\Sigma_{(M,0)}$ and $\w - \Sigma_N$ intersect at $\sigma_{\FF}$ and $\w - \sigma_{\GG}$. Then there exists a $\beta$-nbc-basis $B$ of $M$ such that $\FF = \FF_M(B)$ and $\mathcal G = \FF_N(B^\perp)$.
\end{lm}

\begin{proof}
If $\sigma_{\FF}$ and $\w - \sigma_{\GG}$ intersected at more than one point, their intersection would contain a line segment, so $\Sigma_{(M,0)} \cap (\w - \Sigma_N)$ would be infinite. Since $\Sigma_{(M,0)}$ and $- \Sigma_N$ have complementary dimensions, this would contradict the genericity of $\w$. 

Therefore $\sigma_{\FF} \cap (\w - \sigma_{\GG})$ is a point, and
Lemma \ref{lm:arboreal} implies that  the set compositions $\opi$ and $\otau$ of $\FF$ and $\GG$ form an arboreal pair; that is, $\Gamma_{\opi, \otau}$ is a tree. In particular, $\pi_a \cap \tau_b = (F_a - F_{a-1}) \cap (G_b - G_{b-1})$ cannot have more than one element for any $a$ and $b$. We proceed in several steps.

\medskip

1. Our first step will be to show that in the intersection tree $\Gamma_{\opi, \otau}$, the top right vertex $\pi_{r+1}$ contains $0$ and $1$, the bottom right vertex $\tau_{n-r}$ contains $1$, and thus the edge $1$ connects these two rightmost vertices.

The matroid $N = (M/0)^\perp = M^\perp - 0$ can be obtained by deleting the element $0$ from the matroid $M^\perp$.
Each $G_i$ is a flat of $N$, so $G_i^{\bullet} \coloneqq \cl_{M^{\perp}}(G_i) \in \{G_i, G_i \cup 0\}$ is a flat of $M^\perp$. Consider the flag of flats of $M^\perp$
\[ 
\GG^{\bullet} \coloneqq \{ \emptyset = G_0^{\bullet} \subsetneq G_1^{\bullet} \subsetneq \dots \subsetneq G_{n-r-1}^{\bullet} \subsetneq G_{n-r}^{\bullet} = E\},  
\]
where $G_{n-r}^{\bullet} = E$ because $0$ is not a coloop of $M^\perp$ and $G_0^{\bullet} = \emptyset$ because $0$ is not a loop of $M^\perp$. 
Let $m$ be the minimal index such that $0 \in G_m^\bullet$, so
\[ 
\GG^{\bullet} \coloneqq \{ \emptyset = G_0 \subsetneq G_1 \subsetneq \dots \subsetneq G_{m-1} \subsetneq G_m \cup 0 \subsetneq \dots \subsetneq G_{n-r-1} \cup 0 \subsetneq G_{n-r} \cup 0 = E\},  
\]

Consider the unions of the flat $F_r$ with the coflats in $\GG^\bullet$; let $j$ be the index such that
\[
F_r \cup G_{j-1}^{\bullet} \neq E, \qquad F_r \cup G_j^{\bullet} = E
\]
The former cannot have size $|E|-1$ because it is the union of a flat and a coflat. Therefore
\begin{equation}\label{eq1}
(F_r \cup G_j^{\bullet}) - (F_r \cup G_{j-1}^{\bullet}) = (E-F_r) \cap (G_j^{\bullet} - G_{j-1}^{\bullet}) \text{ has size at least $2$.}
\end{equation}
But $\FF$ and $\GG$ are arboreal so 
\begin{equation}\label{eq2}
\pi_{r+1} \cap \tau_j = (E-F_r) \cap (G_j - G_{j-1}) \text{ has size at most $1$}.
\end{equation}
Now observe that $G_j^\bullet-G_{j-1}^\bullet$ and $G_j-G_{j-1}$ can only differ by $\{0\}$, so \eqref{eq1} and \eqref{eq2} imply that they \textbf{must} differ by $\{0\}$; furthermore, the differing element 0 must be in $E-F_r$. We conclude: 

\smallskip

a) $G_j^{\bullet} = G_j \cup 0$ and $G_{j-1}^{\bullet} = G_{j-1}$, that is, $j=m$.
\smallskip

b) $0 \in E-F_r = \pi_{r+1}$.
%
%
%
%
%
%
%
%
%
%
%
\smallskip

Similarly, consider the union of the coflat $G_{n-r-1}^{\bullet}$ with the flats in $\FF$; let $i$ be the unique index such that
\[
F_{i-1} \cup G_{n-r-1}^{\bullet} \neq E, \qquad F_i \cup G_{n-r-1}^{\bullet} = E.
\]
An analogous argument shows that $(F_i - F_{i-1}) \cap (E-G_{n-r-1}^{\bullet})$ has size at least 2, whereas $\pi_i \cap \tau_{n-r} = (F_i - F_{i-1}) \cap (E-0-G_{n-r-1})$ has size at most $1$. This has three consequences:
\smallskip

c) $G_{n-r-1}^\bullet = G_{n-r-1} $, that is, $m=n-r$. 
\smallskip
 
d) $0 \in F_i - F_{i-1}$, which in light of b) implies that $i=r+1$.
\smallskip

e) $(F_i - F_{i-1}) \cap (E-0-G_{n-r-1}) = \pi_{r+1} \cap \tau_{n-r} = \{e\}$ for some element $e \in E-0$. But $e \in \pi_{r+1}$ means that $x_e=0$ is minimum among all $x_i$s for any  $(0,\x) \in \sigma_\FF$, and $e \in \tau_{n-r}$ means that $y_e$ is minimum among all $y_i$s for any  $\y \in \sigma_\GG$ by \Cref{thm:bergman}. Since $\w=\x+\y$ for some such $\x$ and $\y$,  $w_e=x_e+y_e$ is minimum among all $w_i$s, and since $\w$ is rapidly increasing, $e=1$.

\smallskip

It follows that in the intersection tree $\Gamma_{\opi, \otau}$, the top right vertex $\pi_{r+1}$ contains $0$ and $1$ by d) and e), the bottom right vertex $\tau_{n-r}$ contains $1$ by e), and thus $1$ connects them.

\medskip

2.
Next we claim that for any path in the tree $\Gamma_{\opi, \otau}$ directed towards and ending at edge $1$, the first edge has the largest label.\footnote{This implies that the edge labels decrease along any such path, but we will not use this in the proof.}
Assume contrariwise, and consider a containment-minimal path $P$ that does not satisfy this property; its edges must have labels satisfying $e<f>f_2 > \cdots > f_k$ sequentially.
If edge $e$ goes from $\pi(e)$ to $\tau(e)$, Lemma \ref{lm:x,y} gives $x_e = w_e-w_f \pm \text{(terms smaller than $w_f$)} \approx -w_f < 0 = x_1$, contradicting that $(0,\x) \in \sigma_\FF$. If $e$ goes from $\tau(e)$ to $\pi(e)$, we get $y_e = w_e-w_f \pm \text{(terms smaller than $w_f$)} \approx -w_f < w_1 = y_1$, contradicting that $\y \in \sigma_\GG$.

\medskip

3. Now define
\begin{eqnarray*}
b_i \coloneqq \min (F_i - F_{i-1}) &\text{ for }& i=1, \dots, r+1, \\
c_j \coloneqq \min (G_j - G_{j-1}) &\text{ for }& j = 1, \ldots, n-r.
\end{eqnarray*}
Then $B \coloneqq \{b_1, \ldots, b_{r+1}\}$ and $C \coloneqq \{c_1, \ldots, c_{n-r}\}$ are bases of $M$ and $N$, and $\FF = \FF_M(B)$ and $\GG = \FF_N(C)$. We will show that $B$ is an $\beta$-nbc-basis and $C=B^\perp$. 

\smallskip

To do so, we first notice that the path from vertex $\pi_i = F_i - F_{i-1}$ (resp. $\tau_j = G_j - G_{j-1}$) to edge $1$ must start with edge $b_i$ (resp. $c_j$). Indeed, if it started with some other (necessarily larger) edge $b' \in F_i - F_{i-1}$, then the path from edge $b_i$ to edge $1$ would include edge $b'$, and hence would not start with the largest edge, contradicting 2.
This has two consequences:

\smallskip
f) The sets $B$ and $C$ are disjoint. If we had $b_i=c_j=e$, then edge $e$, which connects vertices $\pi_i = F_i - F_{i-1}$ and $\tau_j = G_j - G_{j-1}$, would have to be the first edge in the paths from both of these vertices to edge $1$; this is impossible in a tree. We conclude that $B$ and $C$ are disjoint. Since 
$|B|=r+1$ and $|C|=n-r$, we have $C = B^\perp$.

\smallskip

g) For each $i$ we have $x_{b_i} \approx w_{b_i}$, because the path from $\tau_i$ to vertex $0$ -- which is the path from $\tau_i$ to edge $1$, with edge $1$ possibly removed -- starts with the largest edge $b_i$, so  Lemma \ref{lm:x,y} gives $x_{b_i} = w_{b_i} \pm \text{(smaller terms)}  \approx w_{b_i}$. Similarly $y_{c_i} \approx w_{c_i}$. Now, $(0, \x) \in \sigma_\FF$ gives $x_{b_1}>\cdots>x_{b_{r+1}}$, which implies $w_{b_1}>\cdots>w_{b_{r+1}}$, which in turn gives
\[
b_1> \cdots > b_r>b_{r+1}; \qquad  \textrm{ and analogously, } \qquad 
c_1 > \cdots > c_{n-r-1}>c_{n-r}=1.
\]
The former implies that $B$ is nbc-basis in $M$ by \Cref{lm:nbc_flats}. The latter, combined with c), implies that  $c_1 > \cdots > c_{n-r-1}>0$ respectively are the minimum elements of $G^\bullet_1, \ldots, G^\bullet_{n-r-1}, G^\bullet_{n-r}=E$ that they sequentially generate, so $C\cup0 \setminus 1 = B^\perp \cup 0 \setminus 1$ is nbc-basis in $M^\perp$. It follows that $B$ is $\beta$-nbc-basis in $M$.

\smallskip

We conclude that $B$ is $\beta$-nbc-basis in $M$, $\FF = \FF_M(B)$, and $\GG = \FF_N(B^\perp)$, as desired. 
\end{proof}


\begin{proof}[Proof of \Cref{thm:main}.1]
This follows by combining the previous two lemmas.
\end{proof}

%

%

\section{\textsf{Proof of the main theorem via torus-equivariant geometry} \label{sec:4}}

In this section we give a proof of \Cref{thm:main}.2 using the framework of \emph{tautological classes} of matroids of Berget, Eur, Spink, and Tseng.  See \cite{BEST} for details on what follows.
Recall that $M$ is a matroid on $E$ of rank $r+1$.

\medskip
In this framework, one works with the Chow ring of the permutohedral fan $\Sigma_E$, which is the Bergman fan of the \emph{Boolean matroid} on $E$ whose only basis is $E$.
Its lattice of flats is the poset of subsets of $E$, and its set of maximal cones is in bijection with the set $\mathfrak S_E$ of permutations of $E$.
Let $S = \Z[t_i: i\in E]$; we can think of it as the ring of polynomials on $\R^E$ with integer coefficients. Then $S^{\mathfrak S_E}$ is the ring of $|E|!$-tuples of polynomials in $S$, one polynomial $f_\sigma$ for each permutation $\sigma$ of $E$, or equivalently, one polynomial $f_\sigma$ for each chamber $\sigma$ of $\Sigma_E$.\footnote{We caution that $S^{\mathfrak S_E}$ does \emph{not} denote the ring of $\mathfrak S_E$-invariants of $S$, despite notational similarity.}
We are interested in the $|E|!$-tuples for which the function $f: \R^E \rightarrow \R$ given by $f(x)=f_\sigma(x)$ for $x \in \sigma$ is well defined.

The \emph{Chow ring} $A^\bullet(\Sigma_E)$ of $\Sigma_E$ has the following description.

\begin{defin}\label{defin:chow}
Let $A^\bullet_T(\Sigma_E)$ be the subring of 
$S^{\mathfrak S_E}$
defined by
\begin{eqnarray*}
A_T^\bullet(\Sigma_E) &=& \left\{\text{continuous piecewise polynomials with integer coefficients supported on } \Sigma_E \right\} \\
 &=& \left\{ (f_\sigma)_{\sigma \in \mathfrak S_E} \in S^{\mathfrak S_E} \ \middle| \ \begin{matrix}\text{for any $\sigma, \sigma'\in \mathfrak S_E$, the polynomials $f_\sigma$ and $f_{\sigma'}$}\\ \text{agree as functions on  $\sigma \cap \sigma' \subseteq \R^E$}\end{matrix}\right\}. 
\end{eqnarray*}
Let $I$ 
be the ideal of $ A_T^\bullet(\Sigma_E)$ generated by the global linear functions. Then
\[
A^\bullet(\Sigma_E) = A_T^\bullet(\Sigma_E)/I.
\]
\end{defin}

One can associate to the fans $\Sigma_{(M,e)}$ and $-\Sigma_{(M/e)^\perp}$ certain elements $[\Sigma_{(M,e)}]$ and $[-\Sigma_{(M/e)^\perp}]$ of $A^\bullet(\Sigma_E)$ as follows.
First, per \Cref{rem:proj}, the fan $\Sigma_E$ in $\R^E$ has lineality space $\mathbbm 1 \R$, and the quotient fan $\Sigma_E/\mathbbm 1 \R$ has a natural unimodular isomorphirm to the \emph{affine braid fan} $\Sigma_{E,e} = \Sigma_E|_{x_e=0}$ in $\R^{E-e}$, whose $|E|!$ chambers correspond to the possible orders of $\{x_f \, : \,  f \in E-e\} \cup \{0\}$. This is the affine Bergman fan of the Boolean matroid with special element $e$.

Then, the fans $\Sigma_{(M,e)}$ and $-\Sigma_{(M/e)^\perp}$ are subfans of $\Sigma_{E,e}$, and they are tropical fans in the sense that they satisfy the balancing condition (see for instance \cite[Definition 5.1]{AHK18}). Via the theory of Minkowski weights \cite{FS97}, they consequently define elements $[\Sigma_{(M,e)}]$ and $[-\Sigma_{(M/e)^\perp}]$ of the Chow ring $A^\bullet(\Sigma_{E,e}) \cong A^\bullet(\Sigma_E)$. 
Moreover, the ring $A^\bullet(\Sigma_E)$ is equipped with a degree map $\deg: A^\bullet(\Sigma_E) \to \Z$, which agrees with the map $\deg$ in \Cref{thm:main} in the sense that  
\begin{equation}\label{eq:degisdeg}
\deg(\Sigma_{(M,e)} \cap -\Sigma_{(M/e)^\perp}) = \deg_{\Sigma_E}([\Sigma_{(M,e)}] \cdot [-\Sigma_{(M/e)^\perp}]).
\end{equation}
For a survey of these facts, see \cite[Section 4]{Huh18b}, \cite[Section 5]{AHK18}, or \cite[Section 7.1]{BEST}.

\medskip
We now describe how \cite{BEST} provided a distinguished representative in $A_T^\bullet(\Sigma_E)$ of the class $[\Sigma_{(M,e)}] \in A^\bullet(\Sigma_E) = A_T^\bullet(\Sigma_E)/I$, and similarly for the class $[-\Sigma_{(M/e)^\perp}]$.
For a matroid $M$ on $E$, consider the following elements of the rings $A_T^\bullet(\Sigma_E)$ and $A^\bullet(\Sigma_E)$, modeled after the geometry of torus-equivariant vector bundles from realizable matroids.
For each permutation $\sigma\in \mathfrak S_E$, let $B_\sigma(M)$ be the lexicographically first basis of $M$ with respect to the ordering $\sigma(1) < \cdots < \sigma(n)$ of the ground set.

\begin{defin}\cite[Definition 3.9]{BEST}
Let $M$ be a matroid of rank $r+1$ on a ground set $E$ of size $n+1$.
Its \emph{torus-equivariant tautological Chern classes} are the elements $\{c_i^T(\mathcal S_M^\vee)\}_{i=0, \ldots, r+1}$ and $\{c_j^T(\mathcal Q_M)\}_{j = 0, \ldots, n-r}$ in $A_T^\bullet(\Sigma_E)$ defined by
\begin{align*}
c_i^T(\mathcal S_M^\vee)_\sigma &= \text{the $i$-th elementary symmetric polynomial in $\{t_k : k\in B_\sigma(M)\}$} \quad\text{and}\\
c_j^T(\mathcal Q_M)_\sigma &= \text{the $j$-th elementary symmetric polynomial in $\{-t_\ell: \ell\in E\setminus B_\sigma(M)\}$}
\end{align*}
for any permutation $\sigma\in \mathfrak S_E$.  Their images in the quotient $A^\bullet(\Sigma_E)$, denoted $c_i(\mathcal S_M^\vee)$ and $c_j(\mathcal Q_M)$, are called the \emph{tautological Chern classes} of $M$.
\end{defin}

\cite[Proposition 3.8]{BEST} shows that these elements are well-defined.
The results of \cite{BEST} yield the following representatives in $A_T^\bullet(\Sigma_E)$ of the elements $[\Sigma_{(M,e)}]$ and $[-\Sigma_{(M/e)^\perp}]\in A^\bullet(\Sigma_E)$.
Let $M/e \oplus U_{0,e}$ be the matroid on $E$ obtained from $M/e$ by adding back the element $e$ as a loop. This matroid has rank $r$.

\begin{lm}\label{lem:classes}
Let $M$ be a matroid of rank $r+1$ on a ground set $E$ of size $n+1$.
Define elements $[\Sigma_{(M,e)}]^T$ and $[-\Sigma_{(M/e)^\perp}]^T$ in $A_T^\bullet(\Sigma_E)$ by $[\Sigma_{(M,e)}]^T = c_{n-r}^T(\mathcal Q_M)$ and $[-\Sigma_{(M/e)^\perp}]^T = c_{r}^T(\mathcal S^\vee_{M/e \, \oplus \, U_{0,e}})$, or explicitly,
\[
[\Sigma_{(M,e)}]^T_\sigma = \prod_{i \in E\setminus B_\sigma(M)} (-t_i) \quad\text{and}\quad 
[-\Sigma_{(M/e)^\perp}]^T_\sigma = {\prod_{i\in B_\sigma(M/e \, \oplus \, U_{0,e})}} t_i \quad\text{for all $\sigma\in \mathfrak S_E$}.
\]
Then, their images in the quotient $A^\bullet(\Sigma_E)$ are exactly $[\Sigma_{(M,e)}]$ and $[-\Sigma_{(M/e)^\perp}]$, respectively.
\end{lm}

\begin{proof}
The first equality is a restatement of \cite[Theorem 7.6]{BEST} when one notes that the choice of $e\in E$ induces an isomorphism $\R^E/\R(1,\ldots, 1) \simeq \R^{E-e}$.  The second statement also follows from that theorem when one combines it with \cite[Propositions 5.11, 5.13]{BEST}, which describe how tautological Chern classes behave with respect to matroid duality and direct sums, respectively. 
%
\end{proof}


\begin{proof}[Proof of \Cref{thm:main}.2] We begin with
\cite[Theorem 6.2]{BEST} which states that
\[
\deg_{\Sigma_E}\big( [\Sigma_{(M,e)}] \cdot c_{r}(\mathcal S_M^\vee) \big) = \beta(M).
\]
In light of \eqref{eq:degisdeg}, the desired statement $\deg(\Sigma_{(M,e)} \cap -\Sigma_{(M/e)^\perp}) = \beta(M)$ will follow once we show that $[\Sigma_{(M,e)}] \cdot \big(c_{r}(\mathcal S_M^\vee) - [-\Sigma_{(M/e)^\perp}]\big)$ = 0 in $A^\bullet(\Sigma_E)$. 

Towards this end, we look at the distinguished representative of this product in $A_T^\bullet(\Sigma_E)$, and 
show that the variable $t_e$ divides $[\Sigma_{(M,e)}]_\sigma^T \cdot \big(c^T_{r}(\mathcal S_M^\vee)_\sigma- [-\Sigma_{(M/e)^\perp}]^T_\sigma \big)$ for any $\sigma\in \mathfrak S_E$, as follows.

\begin{itemize}
\item If $e\notin B_\sigma(M)$, then $\displaystyle[\Sigma_{(M,e)}]^T_\sigma = \prod_{i \in E\setminus B_\sigma(M)} (-t_i)$ is divisible by $t_e$.
\item If $e\in B_\sigma(M)$, then $B_\sigma(M/e \oplus U_{0,e}) = B_\sigma(M) \setminus e$, and hence 
\begin{eqnarray*}
c^T_{r}(\mathcal S_M^\vee)_\sigma - [-\Sigma_{(M/e)^\perp}]^T_\sigma &=& 
\text{Elem}_{r}(\{t_k \, : \, k \in B_\sigma(M)) - \prod_{j\in B_\sigma(M)\setminus e} t_j  \\  
&=& \sum_{i \in B_\sigma(M)} \Big( \prod_{j\in B_\sigma(M)\setminus i} t_j \Big) -  \prod_{j\in B_\sigma(M)\setminus e} t_j  \\ 
&=& \sum_{i \in B_\sigma(M) \setminus e} \Big( \prod_{j\in B_\sigma(M)\setminus i} t_j \Big)
\end{eqnarray*}
is divisible by $t_e$.
\end{itemize}

This means that $[\Sigma_{(M,e)}]^T \cdot \big( c^T_{r}(\mathcal S_M^\vee) - [-\Sigma_{(M/e)^\perp}]^T \big)$ is a multiple of the global polynomial $t_e$, and hence is in the ideal $I$ of \Cref{defin:chow}. Therefore $[\Sigma_{(M,e)}] \cdot \big( c_{r}(\mathcal S_M^\vee) - [-\Sigma_{(M/e)^\perp}] \big) = 0$ in the quotient $A^\bullet(\Sigma_E)$, as desired.
\end{proof}

\begin{rem}
Since \Cref{thm:main}.2 was established for matroids realizable over $\R$ in \cite{agostini2021likelihood}, one may attempt to give yet another proof of \Cref{thm:main}.2 via the following property of matroid valuations \cite{derksen2010valuative}:
If two functions $f(M)$ and $g(M)$ coincide for all matroids $M$ that are realizable over $\R$, and if the functions $f$ and $g$ are \emph{valuative} under matroid subdivisions \cite[Definition 3.10]{ardila2010valuations}, then $f(M)$ and $g(M)$ coincide for general, not necessarily realizable, matroids $M$.
The right-hand-side of \Cref{thm:main}.2, the beta invariant $\beta(M)$, is valuative \cite{ardila2010valuations}.  For the left-hand-side however, 
while the maps $M \mapsto \Sigma_{(M,e)}$ and $M \mapsto -\Sigma_{(M/e)^\perp}$ are each valuative, products of valuative functions are in general not valuative.  Thus, it is a priori unclear why the map $f: M\mapsto \deg(\Sigma_{(M,e)} \cdot -\Sigma_{(M/e)^\perp})$ is valuative.
We do not know any argument that establishes the valuativity of the left-hand-side of \Cref{thm:main}.2 independently of the theorem.
\end{rem}

\section{\textsf{Acknowledgments}}

We thank Bernd Sturmfels and Lauren Williams for organizing the Workshop on Nonlinear Algebra and Combinatorics from Physics  in April 2022 at the Harvard University Center for the Mathematical Sciences and Applications.
 In particular, Bernd asked the question that gave rise to this project.

\bibliography{bibli.bib}
\bibliographystyle{alpha}

\end{document}